\documentclass[12pt,reqno]{amsart}
\usepackage{amsmath}
\usepackage{amssymb}
\usepackage{stmaryrd}
\usepackage[all]{xy}
\usepackage{a4wide}
\usepackage[utf8]{inputenc}

\makeatletter
\@namedef{subjclassname@2010}{  \textup{2010} Mathematics Subject Classification}
\makeatother
\newtheorem{theorem}{Theorem}[section]
\newtheorem{lem}[theorem]{Lemma}
\newtheorem{thm}[theorem]{Theorem}
\newtheorem{prop}[theorem]{Proposition}
\newtheorem{cor}[theorem]{Corollary}
\theoremstyle{definition}
\newtheorem{defn}[theorem]{Definition}
\newtheorem{ex}[theorem]{Example}
\newtheorem{exs}[theorem]{Examples}

\newtheorem{rem}[theorem]{Remark}
\newtheorem{rems}[theorem]{Remarks}

\newtheorem{punto}[theorem]{}
\newtheorem*{notation}{Notation}

\input{tcilatex}

\begin{document}
\title[Topological Lattices]{On Topological Lattices and an Application to
First Submodules}
\author{Jawad Abuhlail}
\thanks{Corresponding Author: Jawad Abuhlail}
\address{Department of Mathematics and Statistics\\
Box $\#$ 5046, KFUPM, Dhahran, KSA\\
Fax: +96638602340}
\email{abuhlail@kfupm.edu.sa}
\urladdr{http://faculty.kfupm.edu.sa/math/abuhlail/}
\author{Christian Lomp}
\curraddr{Centro de Matematica\\
Universidade do Porto\\
Rua Campo Alegre 687\\
4169-007 Porto, Portugal}
\email{clomp@fc.up.pt}
\urladdr{http://www.fc.up.pt/pessoas/clomp/}
\date{\today }
\subjclass[2010]{Primary 06A15; Secondary 16D10, 13C05, 13C13, 54B99}
\keywords{Topological lattices; Prime Modules, First Submodules, Strongly
Hollow Submodules, Zariski Topology, Dual Zariski Topology}
\thanks{The first author would like to acknowledge the support provided by
the Deanship of Scientific Research (DSR) at King Fahd University of
Petroleum $\&$ Minerals (KFUPM) for funding this work through project
SB101023. The second author is also grateful to the funding by the European
Regional Development Fund through the program COMPETE and by the Portuguese
Government through the FCT - Funda\c{c}\~{a}o para a ci\^{e}ncia e a
Tecnologia under the pro ject PEst-C/MAT/UI0144/2011. }

\begin{abstract}
We introduce the notion of a (\emph{strongly}) \emph{topological lattice} $%
\mathcal{L}=(L,\wedge ,\vee )$ with respect to a subset $X\subsetneqq L;$ a
prototype is the lattice of (two-sided) ideals of a ring $R,$ which is
(strongly) topological with respect to the prime spectrum of $R.$ We
investigate and characterize (strongly) topological lattices. Given a
non-zero left $R$-module $M,$ we introduce and investigate the spectrum $%
\mathrm{Spec}^{\mathrm{f}}(M)$ of \textit{first submodules} of $M.$ We
topologize $\mathrm{Spec}^{\mathrm{f}}(M)$ and investigate the algebraic
properties of $_{R}M$ by passing to the topological properties of the
associated space.
\end{abstract}

\maketitle

\section{Introduction}

Yassemi \cite{Yas2001} introduced the notion of \textit{second }(\textit{sub}%
)\textit{modules }of a given non-zero module over a commutative ring. This
notion was studied for modules over arbitrary associative rings by Annin
\cite{Ann2002}, where a \textit{second module} was called a \textit{coprime
module}. Moreover, the notion of \textit{coprime submodules} was
investigated by Kazemifard \emph{et al}. \cite{KNR}. In this paper, we
dualize the notion of a coprime submodule to present the spectrum $\mathrm{%
Spec}^{\mathrm{f}}(M)$ of \emph{first submodules} of a given non-zero left
module $M$ over an arbitrary associative, not necessarily commutative, ring $%
R$ with unity. We topologize this spectrum to obtain a dual Zariski-like
topology, study properties of the resulting topological space and
investigate the interplay between the properties of that space and the
algebraic properties of $M$ as an $R$-module.

To achieve this goal, we begin in the second section with a more general
framework of a \emph{topological complete lattice} $\mathcal{L}=(L,\wedge
,\vee ,0,1)$ with respect to a proper subset $X\subsetneqq L$. We
investigate such lattices and characterize them; moreover, we investigate
the irreducibility of the closed subsets of $X.$ In Section 3, we apply the
results we obtained in Section 2 to the concrete example $\mathcal{L}(M),$
the complete lattice of $R$-submodules of a given non-zero $R$-module $M,$
and $X=\mathrm{Spec}^{\mathrm{f}}(M),$ the spectrum of $R$-submodules of $M$
which are prime as $R$-modules. In Section 4, we obtain several algebraic
properties of $_{R}M$ by passing to the topological properties of $\mathrm{%
Spec}^{\mathrm{f}}(M)$.

\section{Topological Lattices}

Throughout, $\mathcal{L}=(L,\wedge ,\vee ,0,1)$ is a complete lattice, $%
X\subseteq L\setminus \{1\}$ is a non-empty subset and $\mathcal{P}=(%
\mathcal{P}(X),\cap ,\cup ,\emptyset ,X)$ is the complete lattice on the
power set of $X.$ We define an order-reversing map%
\begin{equation*}
V:L\longrightarrow \mathcal{P}(X),\qquad a\mapsto V(a)=\{p\in X\mid a\leq
p\}.
\end{equation*}%
It is clear that $V(0)=X$, $V(1)=\emptyset $ and $V\left( \bigvee \mathcal{A}%
\right) =\bigcap_{a\in A}V(a_{i})$ for every $\mathcal{A}\subseteq L$. This
means that the image of $V$ contains $X,$ $\emptyset $ and is closed under
arbitrary intersections. If $\mathrm{Im}(V)$ is also closed under finite
unions, then the elements of $V(L)$ can be considered the closed sets of a
topology on $X$.

\begin{defn}
We say that $\mathcal{L}$ is a \emph{topological}$^{X}$\emph{-lattice} (or $%
X $\emph{-top}, for short) iff $V(L)$ is closed under finite unions.
\end{defn}

The purpose of this section is to characterize $X$-\emph{top lattices}.
Notice that the map $V$ represents the lower adjoint map of a Galois
connection between $\mathcal{L}$ and $\mathcal{P}$, where the upper adjoint
map is%
\begin{equation*}
I:\mathcal{P}(X)\longrightarrow L,\text{ }\mathcal{A}\mapsto \bigwedge
\mathcal{A}.
\end{equation*}%
Since $V,$ $I$ are order reversing and $a\leq I(V(a)),$ $\mathcal{A}%
\subseteq V(I(\mathcal{A}))$ hold for all $a\in L,$ $\mathcal{A}\in \mathcal{%
P}(X)$, we conclude that $(V,I)$ is a Galois connection \cite[3.13]{Gra2011}
and that%
\begin{equation}
V=V\circ I\circ V\qquad \text{and}\qquad I=I\circ V\circ I.  \label{eq_VIV}
\end{equation}%
The compositions $I\circ V$ and $V\circ I$ are \emph{closure operators} \cite%
[Lemma 32]{Gra2011} and the closed elements with respect to this Galois
connection are%
\begin{equation*}
\mathcal{C}(L)=\{a\in L\mid a=I(V(a))\}=\{I(\mathcal{A})\mid \mathcal{A}%
\subseteq X\}=\mathrm{Im}(I)
\end{equation*}%
and%
\begin{equation*}
\mathcal{C}(\mathcal{P}(X))=\{\mathcal{A}\in \mathcal{P}(X)\mid \mathcal{A}%
=V(I(\mathcal{A}))\}=\{V(a)\mid a\in L\}=\mathrm{Im}(V).
\end{equation*}%
Clearly, $V$ is a bijection between $\mathcal{C}(L)$ and $\mathcal{C}(%
\mathcal{P}(X))$ with inverse $I$.

\subsection*{A lattice structure on $\mathcal{C}(L)$}

Note that $X\subseteq \mathcal{C}(L)$, because for every element $p\in X$ we
have $I(V(p))=\bigwedge ([p,1[\cap X)=p$. Moreover, $(\mathcal{C}(L),\wedge
,\bigwedge X)$ is a complete lower semilattice because if $Y\subseteq
\mathcal{C}(L),$ then for each $y\in Y$ we have $y=I(\mathcal{A}_{y})$ for
some subset $A_{y}\subseteq X$ and it follows that%
\begin{equation*}
\bigwedge Y=\bigwedge_{y\in Y}\bigwedge \mathcal{A}_{y}=\bigwedge
\bigcup_{y\in \mathcal{A}_{y}}\mathcal{A}_{y}=I(\bigcup_{y\in Y}\mathcal{A}%
_{y})\in \mathcal{C}(L).
\end{equation*}%
This makes $\mathcal{C}(L)$ a complete lattice by defining a \emph{new} join
for each subset $Y\subseteq \mathcal{C}(L)$ as%
\begin{equation*}
\widetilde{\bigvee }Y:=IV(\bigvee Y)=\bigwedge \{c\in \mathcal{C}(L)\mid
y\leq c\text{ }\forall \text{ }y\in Y\}.
\end{equation*}%
Notice that this new join $\widetilde{\vee }$ is usually \emph{different}
from the original join $\vee $ of $L$.

\bigskip

Before we characterize $X$-top lattices, we need to recall the following
definition (see for example \cite[Definition 1.1.]{AL2013}). An element $p$
in a lower semilattice $(L,\wedge )$ is called \textit{irreducible} iff for
all $a,b\in L$ with $p\leq a,b$:
\begin{equation}
a\wedge b\leq p\qquad \Rightarrow \qquad a\leq p\text{ or }b\leq p.
\label{equation_1}
\end{equation}%
The element $p$ is called \textit{strongly irreducible} iff Equation (\ref%
{equation_1}) holds for all $a,b\in L$.

\begin{thm}
\label{characterisation_xtop}The following statements are equivalent:

\begin{enumerate}
\item[(a)] $\mathcal{L}$ is an $X$-top lattice;

\item[(b)] $V:(\mathcal{C}(L),\wedge ,\widetilde{\vee })\rightarrow (%
\mathcal{P}(X),\cap ,\cup )$ is an anti-homomorphism of lattices;

\item[(c)] every element $p\in X$ is strongly irreducible in $(\mathcal{C}%
(L),\wedge )$;

\item[(d)] $(\mathcal{C}(L),\wedge ,\widetilde{\vee })$ is a distributive
lattice and every element $p\in X$ is irreducible in $(\mathcal{C}(L),\wedge
).$
\end{enumerate}
\end{thm}

\begin{proof}
$(a)\Rightarrow (b)$ Suppose that $\mathcal{L}$ is $X$-top, \emph{i.e.} $%
V(L) $ is closed under finite unions. Let $a,b\in \mathcal{C}(L)$. By
assumption, $V(a)\cup V(b)=V(c)$ for some $c\in L$. Hence%
\begin{equation*}
a\wedge b=I(V(a))\wedge I(V(b))=I(V(a)\cup V(b))=I(V(c))
\end{equation*}%
and it follows that $V(a\wedge b)=V(I(V(c)))\overset{\text{(\ref{eq_VIV})}}{=%
}V(c)=V(a)\cup V(b)$. Moreover, it is clear that $V(a$ $\widetilde{\vee }$ $%
b)=V(a)\cap V(b)$ for all $a,b\in \mathcal{C}(L).$

$(b)\Rightarrow (c)$ Let $p\in X$ and $a,b\in \mathcal{C}(L)$. If $a\wedge
b\leq p,$ then $V(p)\subseteq V(a\wedge b)=V(a)\cup V(b)$ whence $p\in V(a)$
or $p\in V(b)$, \emph{i.e.} $a\leq p$ or $b\leq p.$

$(c)\Rightarrow (a)$ Let $V(a)$ and $V(b)$ be two closed sets. By Equation (%
\ref{eq_VIV}), we can write them as $V(a)=V(a^{\prime })$ and $%
V(b)=V(b^{\prime })$ for some $a^{\prime },b^{\prime }\in \mathcal{C}(L)$.
Let $p\in V(a^{\prime }\wedge b^{\prime })$, whence $a^{\prime }\wedge
b^{\prime }\leq p.$ Since $p$ is strongly irreducible in $\mathcal{C}(L)$, $%
a^{\prime }\leq p$ or $b^{\prime }\leq p$, \emph{i.e.} $p\in V(a^{\prime })$
or $p\in V(b^{\prime })$. Thus $V(a^{\prime }\wedge b^{\prime })\subseteq
V(a)\cup V(b)$. Since $V(a)\cup V(b)=V(a^{\prime })\cup V(b^{\prime
})\subseteq V(a^{\prime }\wedge b^{\prime })$ always holds, the equality
follows.

$(d)\Rightarrow (c)$ holds by \cite[Lemma 1.20]{AL2013}.

$(b+c)\Rightarrow (d)$ Note that $V:\mathcal{C}(L)\rightarrow \mathcal{P}(X)$
is injective and, by $(b),$ the dual lattice $\mathcal{C}(L)^{\circ }$ is
isomorphic to a sublattice of the distributive lattice $\mathcal{P},$ whence
$(\mathcal{C}(L),\wedge ,\widetilde{\vee })$ is distributive as well. On the
other hand, every strongly irreducible element is in particular irreducible.
\end{proof}

\begin{ex}
\label{example_zariski_prime}Let $R$ be an associative, not necessarily
commutative, ring with unity, $X=\mathrm{Spec}(R)$ be the spectrum of prime
ideals of $R$ and $\mathcal{L}_{2}(R)$ the lattice of ideals of $R.$ Notice
that $\mathrm{Im}(I)$ consists of all ideals that are intersections of prime
ideals, \emph{i.e.} the \emph{semiprime ideals} of $R$ \cite[2.5]{Wis1991}.
It is clear that every prime ideal $P$ is strongly irreducible in $\mathcal{L%
}_{2}(R)$; in particular, $P$ is strongly irreducible in $\mathrm{Im}(I)$
whence $\mathcal{L}_{2}\left( {R}\right) $ is a $\mathrm{Spec}(R)$-top
lattice. The topology on $\mathrm{Spec}(R)$ is the ordinary Zariski topology.
\end{ex}

\begin{defn}
We say that $\mathcal{L}$ is a \textit{strongly }$X$\textit{-top} lattice
(or \textit{strongly }$X$\textit{-top} for short) iff every element of $X$
is strongly irreducible in $(L,\wedge )$.
\end{defn}

The proof of the following result is similar to that of Theorem \ref%
{characterisation_xtop}: If all elements $p\in X$ are strongly irreducible
in $(L,\wedge )$, then it follows by Theorem \ref{characterisation_xtop}
that $\mathcal{L}$ is an $X$-top lattice. Moreover, for all $a,b\in L$ we
have%
\begin{equation*}
p\in V(a\wedge b)\text{ }\Rightarrow \text{ }[a\wedge b\leq p\Rightarrow
a\leq p\text{ or }b\leq p]\text{ }\Rightarrow \text{ }p\in V(a)\cup V(b),
\end{equation*}%
\emph{i.e.} $V(a\wedge b)\subseteq V(a)\cup V(b).$ The reverse inclusion is
obvious; this means that $V(a\wedge b)=V(a)\cup V(b)$ for all $a,b\in L.$ On
the other hand, it is clear that $V(a\vee b)=V(a)\cap V(b)$ for all $a,b\in
L.$

\begin{prop}
The following statements are equivalent:

\begin{enumerate}
\item[(a)] $\mathcal{L}$ is a strongly $X$-top lattice;

\item[(b)] $V:\mathcal{L}\rightarrow \mathcal{P}$ is an anti-homomorphism of
lattices.
\end{enumerate}
\end{prop}

\begin{ex}
\label{example_not_strongly_top}Let $R$ be an arbitrary associative ring
with unity and $X=\mathrm{Spec}(R).$ As mentioned in Example \ref%
{example_zariski_prime}, every prime ideal is strongly irreducible in $%
\mathcal{L}_{2}(R)$. In particular, if $R$ is commutative (or more generally
\emph{left duo}), then the lattice $\mathcal{L}\left( {_{R}R}\right) $ of
left ideals of $R$ is strongly $X$-top. However, if $\mathcal{L}_{2}(R)\neq
\mathcal{L}\left( {_{R}R}\right) $, then $\mathcal{L}\left( {_{R}R}\right) $
might not be strongly $X$-top. For example, if $R$ is a prime ring which is
not uniform as a left $R$-module, then $\mathcal{L}\left( {_{R}R}\right) $
is not strongly $X$-top because $P=0$ is a prime ideal and there are
non-zero left ideals $A$, $B$ of $R$ with $A\cap B=0$. An example of such a
ring is given by the full $n\times n$-matrix ring $R=\mathrm{M}_{n}(K)$ over
a field $K$ where $n\geq 2$.
\end{ex}

Recall from \cite{Bou1966} that for a non-empty topological space $\mathbf{X}
$, a non-empty subset ${\mathcal{A}}\subseteq {\mathbf{X}}$ is said to be
\textit{irreducible }in $\mathbf{X}$ iff for all proper closed subsets $%
A_{1},A_{2}$ of $\mathbf{X}$ we have%
\begin{equation*}
{\mathcal{A}}\subseteq A_{1}\cup A_{2}\Rightarrow {\mathcal{A}}\subseteq
A_{1}\text{ or }{\mathcal{A}}\subseteq A_{2}.
\end{equation*}%
A maximal irreducible subset of $\mathbf{X}$ is called an \emph{irreducible
component} and is necessarily closed.

\begin{prop}
\label{irreducible_subsets}Let $\emptyset \neq \mathcal{A}\subseteq X$.

\begin{enumerate}
\item[(1)] Let $\mathcal{L}$ be $X$-top. If $I(\mathcal{A})$ is irreducible
in $(\mathcal{C}(L),\wedge )$, then $\mathcal{A}$ is an irreducible subset
of $X$.

\item[(2)] Let $\mathcal{L}$ be strongly $X$-top. The following are
equivalent:

\begin{enumerate}
\item $I(\mathcal{A})$ is irreducible in $(\mathcal{C}(L),\wedge );$

\item $\mathcal{A}$ is an irreducible subset of $X;$

\item $I(\mathcal{A})$ is \emph{(}strongly\emph{)} irreducible in $(L,\wedge
).$
\end{enumerate}
\end{enumerate}
\end{prop}

\begin{proof}
(1) By our assumption, $X$ becomes a topological space. Suppose that $%
\mathcal{A}\subseteq V(a_{1})\cup V(a_{2})$ for some $a_{1},a_{2}\in L$. Set
$\mathcal{A}_{i}=V(a_{i})\cap \mathcal{A}$ for $i=1,2$, so that $\mathcal{A}=%
\mathcal{A}_{1}\cup \mathcal{A}_{2}.$ Notice that $I(\mathcal{A})=I(\mathcal{%
A}_{1})\wedge I(\mathcal{A}_{2}),$ whence $I(\mathcal{A})=I(\mathcal{A}_{i})$
for some $i=1,2$ as $I(\mathcal{A})$ is assumed to be irreducible in $%
\mathcal{C}(L),$ and it follows that%
\begin{equation*}
\mathcal{A}\subseteq V(I(\mathcal{A}))=V(I(\mathcal{A}_{i}))\subseteq
V(I(V(a_{i}))=V(a_{i}).
\end{equation*}

(2) Suppose that all elements of $X$ are strongly irreducible in $(L,\wedge
).$

$(a)\Rightarrow (b)$ follows by (1).

$(b)\Rightarrow (c)$ Let $\mathcal{A}$ be an irreducible subset of $X$ and
assume that $a_{1}\wedge a_{2}\leq I(\mathcal{A})$ for some $a_{1},a_{2}\in
L $. It follows that
\begin{equation*}
\mathcal{A}\subseteq V(I(\mathcal{A}))\subseteq V(a_{1}\wedge
a_{2})=V(a_{1})\cup V(a_{2}).
\end{equation*}%
As $\mathcal{A}$ is irreducible, $\mathcal{A}\subseteq V(a_{i})$ for some $%
i=1,2,$ whence $I(\mathcal{A})\geq I(V(a_{i}))\geq a_{i}$ showing that $I(%
\mathcal{A})$ is strongly irreducible in $(L,\wedge )$.

$(c)\Rightarrow (a)$ is obvious.
\end{proof}

\begin{ex}
Let $R$ be a simple ring and $X=\mathrm{Spec}(R)=\{0\}.$ Clearly, $\mathcal{L%
}\left( {_{R}R}\right) $ is an $X$-top lattice. Notice that $X$ is
irreducible since it is a singleton. However, $I(X)=0$ is irreducible in $(%
\mathcal{L}\left( {_{R}R}\right) ,\cap )$ if and only if $R$ is uniform as
left $R$-module if and only if $\mathcal{L}\left( {_{R}R}\right) $ is
strongly $X$-top. Thus, every simple ring that is not left uniform can be
taken as an example to show that the hypothesis on $\mathcal{L}$ to be
strongly $X$-top in \ref{irreducible_subsets} (2) cannot be dropped.
\end{ex}

\begin{cor}
\label{x->irred}If $\mathcal{L}$ is $X$-top and $\mathcal{A}\subseteq X$ is
such that $I(\mathcal{A})\in X$, then $\mathcal{A}$ is irreducible.
\end{cor}

The following result will be needed when dealing with \emph{first submodules}%
.

\begin{cor}
\label{uniserial_xtop}Let $\mathcal{L}$ be $X$-top. If $[x,1[\subseteq X$
for some $x\in X$, then $[x,1[$ is a chain. Moreover, if $[x,1[\subseteq X$
for every $x\in X$, then every non-empty subset $\mathcal{A}\subseteq X$
with $I(\mathcal{A})\in X$ is a chain.
\end{cor}

\begin{proof}
Let $x\in X$ be such that $[x,1[\subseteq X$ and $a,b\in L$ be such that $%
x\leq a,b.$ By hypothesis, $a,b$ and $c:=a\wedge b$ belong to $X$. Thus, by
Theorem \ref{characterisation_xtop}, $c$ is strongly irreducible in $(%
\mathcal{C}(L),\wedge )$, \emph{i.e.} $a=c$ or $b=c$. Hence, $a\leq b$ or $%
b\leq a$. Assume now that $[x,1[\subseteq X$ for every $x\in X$ and let $%
\mathcal{A}\subseteq X$ be a non-empty subset. If $I(\mathcal{A})\in X$,
then $[I(\mathcal{A}),1[$ is a chain and hence $\mathcal{A}\subseteq \lbrack
I(\mathcal{A}),1[$ is a chain as well.
\end{proof}

\begin{ex}
Let $R$ be an associative, not necessary commutative, ring with unity and $X=%
\mathrm{Max}(R)$ the spectrum of maximal ideals of $R$. The lattice $%
\mathcal{L}_{2}\left( {R}\right) $ of all ideals of $R$ is clearly strongly $%
X$-top. If $R$ has the property that every ideal is contained in a unique
maximal ideal (\emph{e.g.} $R$ is local), then every closed set, in
particular every \emph{connected component}, is a singleton whence $X$ is
totally disconnected.
\end{ex}

\begin{ex}
Let $X=\mathrm{Max}(_{R}R)$ be the spectrum of maximal left ideals of $R.$
The lattice $\mathcal{L}\left( _{R}{R}\right) $ of left ideals of $R$ is not
strongly $X$-top (\emph{cf.} \cite[Example 2.12]{AL2013}).
\end{ex}

\section{First Submodules}

Throughout, $R$ is an associative, not necessarily commutative, ring with
unity, $M$ is a non-zero left $R$-module, $\mathcal{L}(M)=(\mathrm{Sub}%
(M),\cap ,+,0,M)$ is the complete lattice of $R$-submodules of $M$ and $%
\mathcal{S}(M)$ is the (possibly empty) class of simple submodules of $M$.
Moreover, $\mathbf{P}=\{2,3,5,7,\cdots \}$ is the set of all prime positive
integers.

\subsection*{Prime modules}

Recall from \cite{GW2004} the following definition: $_{R}M$ is \emph{fully
faithful} iff every non-zero $R$-submodule of $M$ is faithful. Moreover,
call $_{R}M$ a \emph{prime module} iff $M$ is a non-zero fully faithful $R/%
\mathrm{ann}_{{R}}({M})$-module (see \cite[p.48]{GW2004}). It is easy to see
that $\mathrm{ann}_{{R}}({M})$ is a prime ideal if $M$ is prime module (see
\cite[Exercise 3I]{GW2004}). For every prime ideal $P$ of $R$, the cyclic
left $R$-module $M=R/P$ is a left prime module, because if $N=I/P$ is any
non-zero left $R$-submodule of $M$ with $I$ a left ideal of $R$ properly
containing $P$, then $\mathrm{ann}_{{R}}({N})I\subseteq P$, \emph{i.e.} $%
\mathrm{ann}_{{R}}({N})\subseteq P=\mathrm{ann}_{{R}}({M})$. The class of
left prime $R$-modules is denoted by $\mathbb{P}$ and is clearly closed
under non-zero submodules.

\subsection*{Prime submodules}

We call a proper submodule $N$ of $M$ a \emph{prime submodule} iff $M/N\in
\mathbb{P}.$ Taking%
\begin{equation*}
X=\mathrm{Spec}^{\mathrm{p}}(M)=\{N\in \mathcal{L}\left( {M}\right) \mid N%
\text{ is a prime submodule of }M\},
\end{equation*}%
one defines $M$ to be a \emph{top}$^{\mathrm{p}}$\emph{-module} iff $%
\mathcal{L}\left( {M}\right) $ is $X$-top (\emph{cf.} \cite{MMS1998}). There
are other choices to topologize certain subsets of $\mathcal{L}\left( {M}%
\right) $. For instance, one could take $X=\mathrm{Spec}^{\mathrm{fp}}(M),$
the class of \emph{fully prime} submodules \cite{Abu2011-a} or $X=\mathrm{%
Spec}^{\mathrm{fc}}(M)$ the class of \emph{fully coprime} submodules \cite%
{Abu2011-b}. Other choices are $X=\mathrm{Spec}^{\mathrm{c}}(M)$ the class
of coprime submodules$,$ or $X=\mathrm{Spec}^{\mathrm{s}}(M)$ the class of
\emph{second} submodules \cite{Abu}. For other possible choices for $X,$ see
the (co)primeness notions in the sense of Bican et al. \cite{BJKN80}.

\subsection*{First submodules}

In this work, we are interested in the set $X$ of those submodules of $M$
which belong to $\mathbb{P},$ \emph{i.e.} those which are, \emph{as modules}%
, prime. We set%
\begin{equation*}
\mathrm{Spec}^{\mathrm{f}}({M}):=\mathbb{P}\cap \mathcal{L}\left( {M}\right)
\end{equation*}%
and call its elements \emph{first submodules} of $M.$ We say that $_{R}M$ is
\emph{firstless}\textit{\ }iff $\mathrm{Spec}^{\mathrm{f}}(M)=\emptyset .$

The following proposition can be easily proved and includes some
characterizations of first submodules that will be used in the sequel; more
characterizations can be derived from \cite[1.22]{Wij2006}.

\begin{prop}
\label{IM}The following are equivalent for a non-zero $R$-submodule $0\neq
F\leq _{R}M.$

\begin{enumerate}
\item $F\leq _{R}M$ is a first submodule;

\item $\mathrm{ann}_{R}(F)=\mathrm{ann}_{R}(H)$ for every non-zero submodule
$0\neq H\leq _{R}M;$

\item $\mathrm{ann}_{R}(F)=\mathrm{ann}_{R}(H)$ for every non-zero fully
invariant submodule $0\neq H\leq _{R}^{\mathrm{f.i.}}M;$

\item every non-zero fully invariant submodule of $F$ is a first submodule;

\item every non-zero submodule of $F$ is a first submodule;

\item For every $r\in R$ and $f\in F$ we have:%
\begin{equation}
rRf=0\Rightarrow f=0\text{ or }rF=0.  \label{first-char}
\end{equation}
\end{enumerate}
\end{prop}

Recall that one calls $_{R}M$ is \emph{colocal} (or \emph{cocyclic} \cite%
{Wis1991}) iff the intersection of all non-zero submodules of $M$ is
non-zero.

\begin{rem}
\label{S(M)}If $0\neq F\leq _{R}M$ is simple, then $F$ is indeed a first $R$%
-submodule (\emph{i.e. }$\mathcal{S}(M)\subseteq \mathrm{Spec}^{\mathrm{f}%
}(M)$). So, if $_{R}M$ has an essential socle (called also \emph{atomic }%
\cite{HS2010}), then $\mathrm{Spec}^{\mathrm{f}}({M})\neq \emptyset .$
\end{rem}

\begin{ex}
\label{max-first}Let $0\neq F\leq _{R}M.$ If $\mathrm{ann}_{R}(F)\in \mathrm{%
Max}(R),$ then $_{R}F$ is first in $M:$ if $H\leq _{R}F$ is such that $%
\mathrm{ann}_{R}(H)F\neq 0,$ then $\mathrm{ann}_{R}(F)+\mathrm{ann}_{R}(H)=R$
whence $H=(\mathrm{ann}_{R}(F)+\mathrm{ann}_{R}(H))H=0.$ It follows that if $%
R$ is a simple ring, then every non-zero $R$-submodule of $M$ is first. In
particular, every non-zero subspace of a left vector space over a division
ring is first.
\end{ex}

\begin{exs}
\label{dr}

\begin{enumerate}
\item[(1)] If $0\neq F\leq _{R}M$ has no non-trivial fully invariant $R$%
-submodules, then $F$ is a first submodule of $M.$ For instance, $\mathbb{Q}%
\leq _{\mathbb{Z}}\mathbb{R}$ is a first submodule since $\mathbb{Q}$ has no
non-trivial fully invariant $\mathbb{Z}$-submodules.

\item[(2)] A non-zero semisimple submodule of $M$ need not be first. In case
$R$ is commutative, a semisimple $R$-submodule of $M$ is first if and only
if it is non-zero and homogeneous semisimple.

\item[(3)] Consider the $\mathbb{Z}$-module $M=\mathbb{Z}\oplus \mathbb{Q}%
\oplus \mathbb{R}$ and $F=\mathbb{Z}\oplus \mathbb{Q}.$ Every fully
invariant $\mathbb{Z}$-submodule of $F$ is of the form $n\mathbb{Z}\oplus
\mathbb{Q}$ for some $n\in \mathbb{N}$ and indeed $\mathrm{ann}_{\mathbb{Z}%
}(n\mathbb{Z}\oplus \mathbb{Q})=(0)=\mathrm{ann}_{\mathbb{Z}}(\mathbb{Z}%
\oplus \mathbb{Q}).$ It follows that $F$ is first in $M.$

\item[(4)] Let $M:=\bigoplus_{n\in \mathbb{N}}\mathbb{Z}/n\mathbb{Z}.$ The $%
\mathbb{Z}$-submodule $F:=\bigoplus_{n\in A}\mathbb{Z}/n\mathbb{Z},$ where $%
A\subseteq \mathbf{P}$ is any \emph{infinite} subset, is \emph{not} a first
submodule since for any $p\in A$ we have $p\mathbb{Z}=\mathrm{ann}_{\mathbb{Z%
}}(\mathbb{Z}/p\mathbb{Z})\neq 0=\mathrm{ann}_{\mathbb{Z}}(F).$

\item[(5)] The \emph{Pr\"{u}fer }$p$-\emph{group}%
\begin{equation*}
\mathbb{Z}_{p^{\infty }}:=\{\frac{n}{p^{k}}+\mathbb{Z}\in \mathbb{Q}/\mathbb{%
Z}\mid n\in \mathbb{Z}\text{ and }k\in \mathbb{N}\}
\end{equation*}%
is \textit{not} first in $\mathbb{Q}/\mathbb{Z}:$ if $H\lvertneqq _{\mathbb{Z%
}}\mathbb{Z}_{p^{\infty }},$ then $H=\mathbb{Z}\{\frac{1}{p^{k}}+\mathbb{Z}%
\} $ for some $k\in \mathbb{N}$ (\emph{e.g.} \cite[17.13]{Wis1991}) whence $%
\mathrm{ann}_{R}(H)\neq 0=\mathrm{ann}_{R}(\mathbb{Z}_{p^{\infty }}).$
\end{enumerate}
\end{exs}

Following \cite[p. 86]{Lam1999}, we call a (prime) ideal of $R$ an \emph{%
associated prime} of $M$ iff $\mathfrak{p}=\mathrm{ann}_{R}(N)$ for some $%
N\in \mathrm{Spec}^{\mathrm{p}}(M);$ the class of associated primes of $M$
is denoted by $\mathrm{Ass}(_{R}M).$ If $R$ is commutative, then $\mathfrak{p%
}\in \mathrm{Ass}(_{R}M)$ if and only if $\mathfrak{p}$ is prime and $%
\mathfrak{p}=(0:_{R}m)$ for some $0\neq m\in M$ (\textrm{e.g.} \cite[Lemma
3.56]{Lam1999}).

\begin{ex}
Let $R$ be a \emph{commutative} ring. If $\mathfrak{p}$ is an associated
prime of $M,$ then $R/\mathfrak{p}\hookrightarrow M$ is a first $R$%
-submodule. Notice that we might not have such an embedding if $R$ is
non-commutative (\emph{e.g.} \cite[Fact 36]{Ann2002}).
\end{ex}

\begin{rem}
\label{cop-prime}If $F\in \mathrm{Spec}^{\mathrm{f}}(M),$ then $\mathrm{ann}%
_{R}(F)$ is a prime ideal: let $I,J\in \mathcal{L}_{2}(R)$ be such that $%
IJ\subseteq \mathrm{ann}_{R}(F)$ and suppose that $J\nsubseteqq \mathrm{ann}%
_{R}(F),$ i.e. $K:=JF\neq 0.$ Since $_{R}F$ is first in $M$ and $%
IK=I(JF)=(IJ)F=0,$ we conclude that $IF=0,$ \textit{i.e.} $I\subseteq
\mathrm{ann}_{R}(F).$ Notice that the converse is not true: for example, $%
\mathrm{ann}_{\mathbb{Z}}(\mathbb{Z}\oplus \mathbb{Z}/8\mathbb{Z})=(0)$ is a
prime ideal of $\mathbb{Z};$ however, $\mathbb{Z}\oplus \mathbb{Z}/8\mathbb{Z%
}$ is not a first $\mathbb{Z}$-submodule of $\mathbb{Z}\oplus \mathbb{Z}/8%
\mathbb{Z}\oplus \mathbb{Z}/3\mathbb{Z}$ since $\mathrm{ann}_{\mathbb{Z}%
}(0\oplus \mathbb{Z}/8\mathbb{Z})=8\mathbb{Z}\neq (0).$
\end{rem}

\section{The topological structure of $\mathrm{Spec}^{\mathrm{f}}({M})$}

Throughout this section, we fix the general setting of Section 3. In
particular, $M$ is a non-zero left $R$-module over the associative unital
ring $R$ and $\mathbb{P}$ is the class of prime $R$-module. An $R$-submodule
$N\leq _{R}M$ is said to be (strongly) hollow iff $N$ is (strongly)
irreducible in $\mathcal{L}(M)^{\circ }=(\mathrm{Sub}(_{R}M),+,\cap ).$ The
class of strongly hollow submodules of $M$ is denoted by $\mathcal{SH}(M).$
In this section, we give some applications of the results in Section 2 to
the dual lattice $\mathcal{L}\left( {M}\right) ^{\circ }.$

\subsection*{Top-modules}

Since $\mathrm{Sub}\left( _{R}{N}\right) \subseteq \mathrm{Sub}\left( _{R}{M}%
\right) $ for every submodule $N$ of $M$, we have $\mathrm{Spec}^{\mathrm{f}%
}({N})\subseteq \mathrm{Spec}^{\mathrm{f}}({M})$. Hence, in order to use the
map $V$ from the second section, we will use the dual lattice $\mathcal{L}%
\left( {M}\right) ^{\circ }$ of $\mathcal{L}\left( {M}\right) $ and $X=%
\mathrm{Spec}^{\mathrm{f}}({M})$. In this case, we have the \emph{%
order-preserving} map%
\begin{equation*}
V:\mathrm{Sub}(_{R}M)\longrightarrow \mathcal{P}(X),\qquad N\mapsto
V(N)=\{P\in \mathbb{P}\mid P\subseteq N\}.
\end{equation*}

The map $V$ forms a Galois connection with the map
\begin{equation*}
I:\mathcal{P}(X)\longrightarrow \mathrm{Sub}(_{R}M),\qquad \mathcal{A}%
\mapsto I(\mathcal{A})=\sum_{P\in \mathcal{A}}P.
\end{equation*}%
As before, we have $V=V\circ I\circ V$ and $I=I\circ V\circ I$. Denote the
image of $V$ by $\xi ^{\mathrm{f}}({M})$. From Section $2,$ we know that $%
\xi ^{\mathrm{f}}({M})$ contains $X,$ $\emptyset $ and is closed under
intersections; note that because of considering the dual lattice of $%
\mathcal{L}\left( {M}\right) $ one has%
\begin{equation*}
\bigcap_{\lambda \in \Lambda }V(N_{\lambda })=V\left( \bigcap_{\lambda \in
\Lambda }N_{\lambda }\right) .
\end{equation*}%
The set $\xi ^{\mathrm{f}}({M})$ can be described as
\begin{equation*}
\xi ^{\mathrm{f}}({M})=\{V(I(\mathcal{A}))\mid \mathcal{A}\subseteq \mathrm{%
Spec}^{\mathrm{f}}({M})\}
\end{equation*}%
and depends only on those submodules that are of the form $I(\mathcal{A})$
for some subset $\mathcal{A}\subseteq \mathrm{Spec}^{\mathrm{f}}({M})$. The
image of $I$ is
\begin{equation*}
\mathcal{I}(M):=\mathcal{C}(\mathcal{L}\left( {M}\right) ^{\circ })=\{I(%
\mathcal{A})\mid \mathcal{A}\subseteq \mathrm{Spec}^{\mathrm{f}}({M})\}
\end{equation*}%
which is the set of closed elements relative to the Galois connection $(V,I)$
and forms an upper subsemilattice $(\mathcal{I}(M),+)$ of $\mathcal{L}\left(
{M}\right) ^{\circ }$. Note that $\mathrm{Spec}^{\mathrm{f}}({M})=\mathbb{P}%
\cap \mathrm{Sub}(_{R}{M)}\subseteq \mathcal{I}(M)$.

\subsection*{A lattice structure on $\mathcal{I}(M)$}

The upper semilattice of closed elements $(\mathcal{I}(M),+)$ is complete,
whence it has a greatest element (which we call the \emph{coradical} of $M$):%
\begin{equation*}
\mathrm{Corad}^{\mathrm{f}}({M})=I(\mathrm{Spec}^{\mathrm{f}}({M}%
))=\sum_{P\in \mathrm{Spec}^{\mathrm{f}}({M})}P.
\end{equation*}%
This allows defining a \emph{new} meet on $\mathcal{I}(M)$ as follows:
consider a family $\{C_{\lambda }\}_{\lambda \in \Lambda },$ where $%
C_{\lambda }=I(\mathcal{A}_{\lambda })$ and $A_{\lambda }\subseteq \mathrm{%
Spec}^{\mathrm{f}}(M)$ for each $\lambda \in \Lambda ,$ and define%
\begin{eqnarray*}
\widetilde{\dbigwedge }C_{\lambda } &=&IV(\bigcap_{\lambda \in \Lambda
}C_{\lambda })=I(\bigcap_{\lambda \in \Lambda }V(C_{\lambda })) \\
&=&\sum \{I(\mathcal{A})\mid I(\mathcal{A})\leq C_{\lambda }\text{ }\forall
\text{ }\lambda \in \Lambda \} \\
&=&\sum \{F\in \mathrm{Spec}^{\mathrm{f}}(M)\mid F\leq \bigcap_{\lambda \in
\Lambda }C_{\lambda }\}.
\end{eqnarray*}%
Notice that this new meet $\widetilde{\wedge }$ is usually \emph{different}
from the original meet $\cap .$

\begin{defn}
We say that $M$ is a

\emph{top}{$^{\mathrm{f}}$}\emph{-module} iff $\mathcal{L}\left( {M}\right)
^{\circ }$ is $\mathrm{Spec}^{\mathrm{f}}(M)$-top, \emph{i.e.} iff $\xi ^{%
\mathrm{f}}({M})$ is closed under finite unions;

\emph{strongly top}$^{\mathrm{f}}$\emph{-module }iff $\mathcal{L}\left( {M}%
\right) ^{\circ }$ is strongly $\mathrm{Spec}^{\mathrm{f}}(M)$-top, \emph{%
i.e.} iff every first submodule of $M$ is strongly hollow.
\end{defn}

From Theorem \ref{characterisation_xtop} and Corollary \ref{uniserial_xtop}
we get

\begin{thm}
\label{characterisation_top}The following statements are equivalent:

\begin{enumerate}
\item[(a)] $M$ is a top$^{\mathrm{f}}$-module;

\item[(b)] $V:(\mathcal{I}(M),\widetilde{\wedge },+)\rightarrow (\xi ^{%
\mathrm{f}}({M}),\cap ,\cup )$ is a lattice isomorphism;

\item[(c)] every first submodule of $M$ is strongly hollow in $\mathrm{Corad}%
^{\mathrm{f}}({M});$

\item[(d)] $(\mathcal{I}(M),\widetilde{\wedge },+)$ is a distributive
lattice and every first submodule of $M$ is a hollow \emph{(}uniserial\emph{)%
} module.
\end{enumerate}
\end{thm}

\begin{proof}
The equivalence follows from Theorem \ref{characterisation_xtop}.

Every $R$-submodule of $P\in \mathrm{Spec}^{\mathrm{f}}({M})$ is also a
prime module, hence $[P,0[\subseteq \mathrm{Spec}^{\mathrm{f}}({M})$ in $%
\mathcal{L}\left( {M}\right) ^{\circ }$. Thus, Corollary \ref{uniserial_xtop}
applies and proves that every $P\in \mathrm{Spec}^{\mathrm{f}}({M})$ is
uniserial.
\end{proof}

\begin{lem}
\label{Steph}If $\mathrm{Soc}(_{R}M)\neq 0,$ then the following are
equivalent:

\begin{enumerate}
\item[(a)] All isomorphic simple submodules of $M$ are equal.

\item[(b)] $\mathrm{Soc}(M)$ is a direct sum of non-isomorphic simple
modules;

\item[(c)] $\mathrm{Soc}(M)$ is distributive;
\end{enumerate}
\end{lem}

\begin{proof}
$(a)\Rightarrow (b)$ this is clear.

$(b)\Longleftrightarrow (c)$ By \cite[Proposition 1.3]{Ste2004}, $\mathrm{Soc%
}(M)=\bigoplus_{\lambda \in \Lambda }E_{\lambda }$ is distributive if and
only if $E_{\alpha }$ and $E_{\beta }$ are \emph{unrelated} for all $\alpha
\neq \beta $ in $\Lambda $; the later means for simple modules that $\mathrm{%
Hom}_{R}(E_{\alpha },E_{\beta })=0$.

$(c)\Rightarrow (a)$ By \cite[Proposition 1.2]{Ste2004}, if $\mathrm{Soc}%
(M)=\bigoplus_{\lambda \in \Lambda }E_{\lambda }$ ($E_{\lambda }$ is simple
for each $\lambda \in \Lambda $) and $E_{\alpha }$ is unrelated to $E_{\beta
}$ for all $\alpha \neq \beta $ in $\Lambda $, then for every submodule $%
X\subseteq \bigoplus_{\lambda \in \Lambda }E_{\lambda }$ one has $%
X=\bigoplus_{\lambda \in \Lambda }(X\cap E_{\lambda })$. In particular, if $%
X $ is simple, then $X=E_{\lambda }$ for some $\lambda \in \Lambda $.
\end{proof}

\begin{cor}
\label{simple-0}If $_{R}M$ is a top$^{\mathrm{f}}$-module, then $\mathrm{Soc}%
(M)$ is a \emph{(}direct\emph{)} sum of non-isomorphic simple modules.
\end{cor}

\begin{proof}
This follows from the fact that $\mathcal{L}(\mathrm{Soc}(M))=(\mathrm{Sub}(%
\mathrm{Soc}(M),\cap ,+))$ is a sublattice of the distributive lattice $(%
\mathcal{I}(M),\widetilde{\wedge },+)$, whence is also distributive. This is
equivalent, by Lemma \ref{Steph}, to the stated property for $\mathrm{Soc}%
(M).$
\end{proof}

\begin{rem}
\label{min}Recall from \cite{Abu2011-b} that $_{R}M$ has the \emph{%
min-property} iff for every simple $R$-submodule $H\leq _{R}M$ we have $%
H\nsubseteqq H_{e},$ where $H_{e}:=\sum\limits_{K\in S(M)\backslash \{H\}}K.$
By Lemma \ref{Steph} and \cite[Theorem 2.3]{Smi2011}, $\mathrm{Soc}(M)$ is
distributive if and only if $_{R}M$ has the min-property.
\end{rem}

\begin{notation}
We set $\mathrm{Sub}_{c}(M):=\{(0:_{M}I)\mid I\in \mathcal{L}_{2}(R)\},$ $%
\mathcal{X}(L):=\mathrm{Spec}^{\mathrm{f}}(M)\backslash V(L)$ and%
\begin{equation*}
\begin{tabular}{lllllll}
$\xi ^{\mathrm{f}}(M)$ & $:=$ & $\{V(L)\mid L\in \mathrm{Sub}(_{R}M)\};$ &
& $\xi _{c}^{\mathrm{f}}(M)$ & $:=$ & $\{V(L)\mid L\in \mathrm{Sub}%
_{c}(M)\}; $ \\
$\tau ^{\mathrm{f}}(M)$ & $:=$ & $\{\mathcal{X}(L)\mid L\in \mathrm{Sub}%
(_{R}M)\};$ &  & $\tau _{c}^{\mathrm{f}}(M)$ & $:=$ & $\{\mathcal{X}(L)\mid
L\in \mathrm{Sub}_{c}(M)\};$%
\end{tabular}%
\end{equation*}
\end{notation}

\begin{rem}
\label{strong->top}Let $M$ be a strongly top$^{\mathrm{f}}$-module.

\begin{enumerate}
\item[(a)] $M$ is a top$^{\mathrm{f}}$-module: this follows directly from
observation that $\mathrm{Spec}^{\mathrm{f}}(M)\subseteq \mathcal{SH}(M)$ if
and only if $V(L_{1})\cup V(L_{2})=V(L_{1}+L_{2})$ for every pair of
submodules $L_{1},L_{2}\leq _{R}M.$

\item[(c)] $\mathrm{Spec}^{\mathrm{f}}(M)$ has a basis of open sets given by
\begin{equation*}
\{\mathcal{X}^{\mathrm{f}}(H)\mid H\leq _{R}M\text{ is finitely generated}\}
\end{equation*}
\end{enumerate}
\end{rem}

\begin{theorem}
\label{f-top}$(\mathrm{Spec}^{\mathrm{f}}(M),\tau _{c}^{\mathrm{f}}(M))$ is
a topological space.
\end{theorem}

\begin{proof}
It is obvious that $V(0)=\emptyset ,$ $V(M)=\mathrm{Spec}^{\mathrm{f}}(M)$
and that $\bigcap\limits_{\lambda \in \Lambda }V(L_{\lambda
})=V(\bigcap\limits_{\lambda \in \Lambda }L_{\lambda })$ for every subset $%
\{L_{\lambda }\}_{\Lambda }\subseteq \mathrm{Sub}_{c}(M)$. We show now that
for all ideals $I,\widetilde{I}$ or $R$ we have
\begin{equation}
V((0:_{M}I))\cup V((0:_{M}\widetilde{I}))=V((0:_{M}I)+(0:_{M}\widetilde{I}%
))=V((0:_{M}I\cap \widetilde{I}))=V((0:_{M}I\widetilde{I})).  \label{L1+L2}
\end{equation}%
Indeed, the following inclusions are obvious%
\begin{equation}
V((0:_{M}I))\cup V((0:_{M}\widetilde{I}))\subseteq V((0:_{M}I)+(0:_{M}%
\widetilde{I}))\subseteq V((0:_{M}I\cap \widetilde{I}))\subseteq V((0:_{M}I%
\widetilde{I}))  \label{IJ-com}
\end{equation}%
On the other hand, let $F\in V((0:_{M}I\widetilde{I}))$ and suppose that $%
F\nsubseteqq (0:_{M}\widetilde{I}).$ Since $I(\widetilde{I}F)=(I\widetilde{I}%
)F=0$ and $\widetilde{I}F\neq 0,$ we conclude that $IF=0$ (recall that $F$
is a first submodule of $M$), \emph{i.e.} $F\subseteq (0:_{M}I).$
Consequently, $F\in V((0:_{M}I))\cup \mathcal{V}^{\mathrm{f}}((0:_{M}%
\widetilde{I})).$
\end{proof}

\begin{ex}
For every non-empty subset $A\subseteq \mathbf{P},$ the $\mathbb{Z}$-module $%
M:=\oplus _{p\in A}\mathbb{Z}/p\mathbb{Z}$ (with no repetition) is a top$^{%
\mathrm{f}}$-module: it can be easily seen that $\mathrm{Spec}^{\mathrm{f}%
}(M)=\{\mathbb{Z}/p\mathbb{Z}\mid p\in A\}$ and that $\xi ^{\mathrm{f}}(M)$
is closed under finite unions.
\end{ex}

\begin{ex}
Let $R=\mathbb{Z}$. The prime $\mathbb{Z}$-modules are the torsionfree
modules and the Abelian $p$-groups, for $p$ a prime number. If $M$ is a
torsion abelian group, then the first submodules are the $p$-subgroups of $M$%
. Since $M=\bigoplus\limits_{p\text{ a prime number}}T_{p}(M)$, where $%
T_{p}(M)$ is the $p$-torsion part of $M$, every submodule of $M$ is a sum of
$p$-subgroups, \emph{i.e.} the lattice $\mathrm{Im}(I)$ is equal to the
whole lattice $\mathcal{L}\left( {M}\right) $ of subgroups of $M$. If $M$ is
a top$^{\mathrm{f}}$-module, then $T_{p}(M)$ has to be uniserial for each $p$
by Theorem \ref{characterisation_top}. On the other hand, if the $p$-torsion
parts $T_{p}(M)$ of $M$ are uniserial, then for every $p$-subgroup $N$ of $M$
contained in a sum of subgroups $H+K$ one has $N\subseteq
T_{p}(H+K)=T_{p}(H)+T_{p}(K)$. Since $T_{p}(M)$ is uniserial, $%
T_{p}(H+K)=T_{p}(H)$ or $T_{p}(H+K)=T_{p}(K)$. Hence $N\subseteq H$ or $%
N\subseteq K$. This shows that a torsion abelian group is a top$^{\mathrm{f}%
} $-module if and only if all its $p$-torsion parts are uniserial. For
instance, $\mathbb{Q}/\mathbb{Z}$ is top$^{\mathrm{f}}$-module.
\end{ex}

\begin{ex}
Over a simple ring $R$, every non-zero left $R$-module is prime. Theorem \ref%
{characterisation_top} shows that the (strongly) top$^{\mathrm{f}}$-modules%
\textit{\ }over a simple ring are precisely the non-zero uniserial modules.
\end{ex}

\begin{rems}
\label{simple-char}Let $M$ be a \textit{top}$^{\mathrm{f}}$\textit{-module,}
$H$ a non-zero submodule of $M$ and set $\mathcal{X}^{\mathrm{f}}(H)=\mathrm{%
Spec}^{\mathrm{f}}({M})\backslash V(H).$

\begin{enumerate}
\item[(a)] $\mathrm{Spec}^{\mathrm{f}}({M})$ is a $T_{0}$ (Kolmogorov) space.

\item[(c)] The closure of any subset $\mathcal{A}\subseteq \mathrm{Spec}^{%
\mathrm{f}}(M)$ is $\overline{\mathcal{A}}=V(I(\mathcal{A})\mathcal{)}.$

\item[(d)] ${\mathcal{X}}(H)=\emptyset $ if and only if $\mathrm{Corad}^{%
\mathrm{f}}({M})\subseteq H$.

\item[(e)] If $_{R}M$ has \emph{essential} socle, then $\mathrm{Spec}^{%
\mathrm{f}}({H})=\emptyset $ if and only if $H=0.$

\item[(f)] $\mathrm{Spec}^{\mathrm{f}}({H})$ is a subspace of $\mathrm{Spec}%
^{\mathrm{f}}({M})$.

\item[(g)] If $M\simeq N$, then $\mathrm{Spec}^{\mathrm{f}}({M})\approx
\mathrm{Spec}^{\mathrm{f}}({N})$ are homeomorphic and $\mathrm{Corad}^{%
\mathrm{f}}({M})\simeq \mathrm{Corad}^{\mathrm{f}}({N})$.
\end{enumerate}
\end{rems}

Recall (\emph{e.g.} \cite{Tug2004}, \cite{A-TF2007}) that $M$ is said to be
a \emph{multiplication} (\emph{comultiplication}) \emph{module} iff every $R$%
-submodule of $M$ is of the form $IM$ ($(0:_{M}I$) for some ideal $I$ of $R,$
or equivalently iff for every $R$-submodule $H\leq _{R}M$ we have $%
H=(H:_{R}M)M$ ($L=(0:_{M}(0:_{R}L)$).

\begin{prop}
\label{c-ann}Let $0\neq F\leq _{R}M.$

\begin{enumerate}
\item[(a)] If $_{R}F$ is comultiplication, then $F$ is first in $M$ if and
only if $_{R}F$ is simple.

\item[(b)] If $_{R}F$ is multiplication, then $F$ is first in $M$ if and
only if $\mathrm{ann}_{R}(F)$ is a prime ideal.
\end{enumerate}
\end{prop}

\begin{proof}
(a) If $_{R}F$ is simple, then $F$ is first in $M$ by Remark \ref{S(M)}. On
the other hand, let $F$ be first in $M,$ $0\neq H\leq _{R}F$ and consider $%
I:=\mathrm{ann}_{R}(H).$ Since $F$ is first in $M,$ we have $I=\mathrm{ann}%
_{R}(F)$ and so $H=(0:_{F}(0:_{R}H))=(0:_{F}(0:_{R}F))=F,$ i.e. $_{R}F$ is
simple.

(b) If $F$ is first in $M,$ then $\mathrm{ann}_{R}(F)$ is a prime ideal by
Remark \ref{cop-prime}. On the other hand, assume that $\mathrm{ann}%
_{R}(F)\in \mathrm{Spec}(R).$ Let $0\neq H\leq _{R}F$ and consider $I:=%
\mathrm{ann}_{R}(H).$ Since $_{R}F$ is multiplication, $H=JF$ for some $J\in
\mathcal{L}_{2}(R).$ Notice that $IJ\subseteq \mathrm{ann}_{R}(F),$ whence $%
IF=0$ since $\mathrm{ann}_{R}(F)$ is a prime ideal and $J\nsubseteqq \mathrm{%
ann}_{R}(F).$ Consequently, $_{R}F$ is first.
\end{proof}

\begin{rem}
\label{zero-dim}Let $R$ be zero-dimensional (\emph{i.e.} every prime ideal
of $R$ is maximal). It follows by Example \ref{max-first} and Remark \ref%
{cop-prime} that
\begin{equation*}
\mathrm{Spec}^{\mathrm{f}}(M)=\{F\leq _{R}M\mid \mathrm{ann}_{R}(F)\text{ is
prime ideal}\}.
\end{equation*}%
Examples of zero-dimensional rings include biregular rings \cite[3.18]%
{Wis1991} and left (right) perfect rings.
\end{rem}

\begin{defn}
Let $0\neq H\leq _{R}M.$ A maximal element of $V(H),$ if any, is said to be
\emph{maximal under }$H.$ A \emph{maximal element of }$\mathrm{Spec}^{%
\mathrm{f}}(M)$ is said to be a \textit{maximal first submodule} of $M.$
\end{defn}

\begin{lem}
\label{s-max}Let $_{R}M$ have an essential socle and

\begin{enumerate}
\item $R$ is zero-dimensional; \emph{or}

\item every submodule of $_{R}M$ is multiplication.
\end{enumerate}

For every $0\neq H\leq _{R}M,$ there exists $F\in \mathrm{Spec}^{\mathrm{f}%
}(M)$ which is maximal under $H.$
\end{lem}

\begin{proof}
Let $0\neq H\leq _{R}M.$ Since $\mathrm{Soc}(M)\leq _{R}M$ is essential, $%
\emptyset \neq \mathcal{S}(H)\subseteq V(H).$ Let%
\begin{equation*}
F_{1}\subseteq F_{2}\subseteq \cdots \subseteq F_{n}\subseteq
F_{n+1}\subseteq \cdots
\end{equation*}%
be an ascending chain in $V(H)$ and set $\widetilde{F}:=\bigcup%
\limits_{i=1}^{\infty }F_{i}.$ Then we have a descending chain of prime
ideals
\begin{equation}
(0:_{R}F_{1})\supseteq (0:_{R}F_{2})\supseteq \cdots \supseteq
(0:_{R}F_{n})\supseteq (0:_{R}F_{n+1})\supseteq \cdots
\end{equation}%
and it follows that $\mathfrak{p}:=(0:_{R}\widetilde{F})=\bigcap%
\limits_{i=1}^{\infty }(0:_{R}F_{i})$ is a prime ideal. If $R$ is
zero-dimensional, then $\widetilde{F}\in \mathrm{Spec}^{\mathrm{f}}(M)$ by
Remark \ref{zero-dim}. On the other hand, if $_{R}\widetilde{F}$ is
multiplication, then $\widetilde{F}\in V(H)$ by Proposition \ref{c-ann} (b).
In either case, it follows by Zorn's Lemma that $V(H)$ has a maximal element.
\end{proof}

\begin{ex}
\label{cc}Recall from \cite[p. 128]{JST2012} that $_{R}M$ is \emph{%
completely cyclic} (or \emph{fully cyclic} \cite{BW2000}) iff every $R$%
-submodule of $_{R}M$ is cyclic. If $_{R}M$ is a uniserial module and $R$ is
a left (or right) Artinian left duo ring, then $_{R}M$ is completely cyclic
by \cite[Lemma 13.9]{JST2012}, whence every $R$-submodule of $_{R}M$ is
multiplication by \cite{Tug2004}; moreover, since $_{R}M$ is cyclic
(finitely generated) and $R$ is left Artinian, $_{R}M$ is also Artinian
whence $\mathrm{Soc}(M)\leq _{R}M$ is essential.
\end{ex}

\begin{punto}
Let $M$ be a \textit{top}$^{\mathrm{f}}$\textit{-module} and consider $%
\mathrm{Spec}^{\mathrm{f}}({M})$ with the associated topology. Since the
lattice $\mathcal{I}(M)$ and the lattice $\xi ^{\mathrm{f}}({M})$ of closed
subsets are isomorphic, some topological conditions on $\mathrm{Spec}^{%
\mathrm{f}}({M})$ translate to module theoretical conditions on $M$. Recall
from \cite{Bou1966, Bou1998} that a topological space is called \emph{%
Noetherian} (\emph{Artinian}) iff every descending (ascending) chain of
closed sets is stationary. Therefore, $\mathrm{Spec}^{\mathrm{f}}({M})$ is
Noetherian (Artinian) if and only if $M$ satisfies the descending
(ascending) chain condition on submodules of the form $I(\mathcal{A})$ for
subsets $\mathcal{A}\subseteq \mathrm{Spec}^{\mathrm{f}}({M})$. In
particular, if $M$ is Noetherian (Artinian), then $\mathrm{Spec}^{\mathrm{f}%
}({M})$ is Artinian (Noetherian).
\end{punto}

\begin{lem}
\label{Lemma_irreducible_enoughprimes}Let $M$ be a \textit{top}$^{\mathrm{f}%
} $\textit{-module}, $\mathcal{A}\subseteq \mathrm{Spec}^{\mathrm{f}}({M})$
an irreducible subset and $H$ a non-zero submodule of $I(\mathcal{A})$. If $%
\mathrm{Spec}^{\mathrm{f}}({H})\neq \emptyset $, then $\mathrm{ann}_{R}(H)=%
\mathrm{ann}_{R}(I(\mathcal{A})).$
\end{lem}

\begin{proof}
Let $P\in \mathrm{Spec}^{\mathrm{f}}({H})$ be a cyclic first submodule.
Setting
\begin{equation*}
\mathcal{A}_{0}=\{Q\in \mathcal{A}\mid Q\cap P=0\},
\end{equation*}%
we have%
\begin{equation*}
\mathcal{A}\subseteq V(I(\mathcal{A}_{0}))\cup V(I(\mathcal{A}\setminus
\mathcal{A}_{0})).
\end{equation*}%
By the irreducibility of $\mathcal{A}$ we have that $\mathcal{A}$ is
contained in one of the two closed sets. Suppose that $\mathcal{A}\subseteq
V(I(\mathcal{A}_{0}))$, whence $P\subseteq I(\mathcal{A}_{0}).$ As $P$ is
cyclic, there is a finite set $\{Q_{1},\cdots ,Q_{n}\}\subseteq \mathcal{A}%
_{0}$ with $P\subseteq Q_{1}+\cdots +Q_{n}$. Since $M$ is a top$^{\mathrm{f}%
} $-module, the lattice of submodules of the form $I(\mathcal{A})$ is
distributive (by Theorem \ref{characterisation_top}). Hence%
\begin{equation*}
P=P\cap (Q_{1}+\cdots +Q_{n})=P\cap Q_{1}+\cdots +P\cap Q_{n}=0,
\end{equation*}%
since $Q_{i}\in \mathcal{A}_{0}$ for all $i=1,\cdots ,n$. This is a
contradiction to $P$ being non-zero. Hence, $\mathcal{A}\subseteq V(I(%
\mathcal{A}\setminus \mathcal{A}_{0}))$ and $P\subseteq I(\mathcal{A})=\sum
\{Q\in \mathcal{A}\mid Q\cap P\neq 0\}.$ This shows that
\begin{equation*}
\mathrm{ann}_{R}(P)\supseteq \mathrm{ann}_{R}(I(\mathcal{A}))=\bigcap_{Q\cap
P\neq 0}\mathrm{ann}_{R}(Q)=\bigcap_{Q\cap P\neq 0}\mathrm{ann}_{R}(Q\cap P)=%
\mathrm{ann}_{R}(P).
\end{equation*}%
Thus $\mathrm{ann}_{R}(P)=\mathrm{ann}_{R}(H)=\mathrm{ann}_{R}(I(\mathcal{A}%
) $.
\end{proof}

\begin{rem}
\label{remark_after_lemma}Note that if $I(\mathcal{A})$ is a distributive
module for a non-empty subset $\mathcal{A}$, then $\mathrm{Spec}^{\mathrm{f}%
}({H})=\emptyset $ if and only if $H=0$ for all submodules $H\in I(\mathcal{A%
})$, because if $H$ is non-zero and $C$ is a non-zero cyclic submodule of $H$%
, then $C\subseteq I(\mathcal{A})$ implies that there are finitely many
first submodules $Q_{1},\ldots ,Q_{n}$ such that $C\subseteq Q_{1}+\cdots
+Q_{n}$. By distributivity, $C=C\cap Q_{1}+\cdots +C\cap Q_{n}$ and since $%
C\neq 0$, there must be some $i=1,\cdots ,n$ such that $C\cap Q_{i}\neq 0$.
Thus $C\cap Q_{i}\in \mathrm{Spec}^{\mathrm{f}}({H})$.
\end{rem}

\begin{prop}
\label{cor_topm}Let $M$ be a top$^{\mathrm{f}}$-module and let $\emptyset
\neq \mathcal{A}\subseteq \mathrm{Spec}^{\mathrm{f}}({M}).$

\begin{enumerate}
\item[(1)] If $I(\mathcal{A})$ is a hollow module, then $\mathcal{A}$ is
irreducible. The converse holds if $M$ is a strongly top$^{\mathrm{f}}$%
-module.

\item[(2)] The following are equivalent:

\begin{enumerate}
\item[(a)] $\mathcal{A}$ is irreducible and $\mathrm{Spec}^{\mathrm{f}}({H}%
)\neq \emptyset $ for any $0\neq H\subseteq I(\mathcal{A})$.

\item[(b)] $\mathcal{A}$ is irreducible and $I(\mathcal{A})$ is distributive.

\item[(c)] $I(\mathcal{A})$ is a first submodule;

\item[(d)] $I(\mathcal{A})$ is uniserial;

\item[(e)] $\mathcal{A}$ is a chain.
\end{enumerate}
\end{enumerate}
\end{prop}

\begin{proof}
(1) follows from Proposition \ref{irreducible_subsets} applied to the dual
lattice $\mathcal{L}(M)^{\circ }.$

(2) $(a)\Rightarrow (c)$. The hypotheses of Lemma \ref%
{Lemma_irreducible_enoughprimes} are fulfilled for any non-zero submodule of
$I(\mathcal{A})$. Hence, all non-zero submodules have the same annihilator,
which shows that $I(\mathcal{A})$ is a prime module.

$(c)\Rightarrow (a)$ By Corollary \ref{x->irred}, $\mathcal{A}$ is
irreducible. Clearly any non-zero submodule of a prime module is first; so,
if $I(\mathcal{A})$ is a first submodule, then any non-zero submodule of it
is first as well.

$(c)\Rightarrow (e)$ follows by Corollary \ref{uniserial_xtop}.

$(e)\Rightarrow (c)$ Assume now that $\mathcal{A}$ is a chain; in
particular, $I(\mathcal{A})=\bigcup_{P\in \mathcal{A}}P$. Since for all $%
Q,P\in \mathcal{A}$ either $Q\subseteq P$ or $P\subseteq Q$ and since $P$
and $Q$ are prime modules, $\mathrm{ann}_{{R}}({P})=\mathrm{ann}_{{R}}({Q})$%
. Every cyclic submodule $U=Rm$ of $I(\mathcal{A})$ lies in one of the
members of $\mathcal{A}$ and thus has the same annihilator, \emph{i.e.} $I(%
\mathcal{A})$ is a prime module or equivalently $I(\mathcal{A})\in \mathrm{%
Spec}^{\mathrm{f}}({M})$.

$(d)\Longleftrightarrow (e)$ clear.

$(a+d)\Rightarrow (b)$ is clear because a uniserial module is distributive.

$(b)\Rightarrow (a)$ holds by Remark \ref{remark_after_lemma}.
\end{proof}

\begin{rem}
\label{rem-semisimple}Let $M$ be a top$^{\mathrm{f}}$-module and $\emptyset
\neq \mathcal{A}\subseteq \mathcal{S}(M).$ Every non-zero submodule of $I(%
\mathcal{A})\subseteq \mathrm{Soc}(M)$ contains a simple (hence first)
submodule and so we get as an immediate consequence from Proposition \ref%
{cor_topm} that the following statements are equivalent:

\begin{enumerate}
\item[(a)] $\mathcal{A}$ is irreducible;

\item[(b)] $I(\mathcal{A})$ is a first submodule of $M;$

\item[(c)] $\mathcal{A}=\{K\}$ as singleton.
\end{enumerate}
\end{rem}

\begin{ex}
Let $M$ be a top$^{\mathrm{f}}$-module. It follows by Remark \ref%
{rem-semisimple} that $\mathcal{S}(M)\subseteq \mathrm{Spec}^{\mathrm{f}}(M)$
is irreducible if and only if $\mathrm{Soc}(M)$ is a first submodule of $M$
if and only if $M$ contains a single simple $R$-submodule.
\end{ex}

\begin{rem}
\label{f->irred}Let $M$ be a top$^{\mathrm{f}}$-module and $\mathcal{A}%
\subseteq \mathrm{Spec}^{\mathrm{f}}({M})$ be such that $I(\mathcal{A})$ is
a first submodule of $M.$ By Theorem \ref{characterisation_top}, $I(\mathcal{%
A})$ is a hollow module (in fact $I(\mathcal{A})$ is moreover a uniserial
module). It follows then from Proposition \ref{cor_topm} (2) that $\mathcal{A%
}$ is irreducible.
\end{rem}

\begin{defn}
We say a top$^{\mathrm{f}}$-module is \emph{consistent} iff for every $%
\mathcal{A}\subseteq \mathrm{Spec}^{\mathrm{f}}(M)$ we have: $I(\mathcal{A}%
)\in \mathrm{Spec}^{\mathrm{f}}(M)$ if (and only if) $\mathcal{A}$ is
irreducible.
\end{defn}

\begin{rem}
From Proposition \ref{cor_topm} and Remark \ref{remark_after_lemma} we see
that the following statements are equivalent for a top$^{\mathrm{f}}$-module
$M$:

\begin{enumerate}
\item[(a)] $M$ is a consistent;

\item[(b)] $\mathrm{Spec}^{\mathrm{f}}(H)\neq \emptyset $ for every non-zero
submodule $H\subseteq I(\mathcal{A})$ and every irreducible subset $\mathcal{%
A}\subseteq \mathrm{Spec}^{\mathrm{f}}(M)$;

\item[(c)] $I(\mathcal{A})$ is distributive for every irreducible subset $%
\mathcal{A}\subseteq \mathrm{Spec}^{\mathrm{f}}(M)$.
\end{enumerate}

For property (c) we use the obvious fact that uniserial modules are
distributive.
\end{rem}

\begin{ex}
Every top$^{\mathrm{f}}$-module with essential socle is consistent.
Moreover, every top$^{\mathrm{f}}$-module $M,$ for which $\mathrm{Corad}^{%
\mathrm{f}}(M)$ is distributive, is consistent.
\end{ex}

\begin{prop}
\label{duo-irr}Let $_{R}M$ be a consistent top$^{\mathrm{f}}$-module with $%
\mathrm{Spec}^{\mathrm{f}}(M)\neq \emptyset $. The following are equivalent
for ${\mathcal{A}}\subseteq \mathrm{Spec}^{\mathrm{f}}(M):$

\begin{enumerate}
\item[(a)] ${\mathcal{A}}$ is irreducible;

\item[(b)] $I({\mathcal{A}})$ is a first submodule of $M;$

\item[(c)] $0\neq I({\mathcal{A}})$ is a hollow module;

\item[(d)] $0\neq I(\mathcal{A})$ is uniserial;

\item[(e)] $\emptyset \neq \mathcal{A}$ is a chain.
\end{enumerate}
\end{prop}

\begin{theorem}
\label{corad-s}Let $_{R}M$ be a consistent top$^{\mathrm{f}}$-module with $%
\mathrm{Spec}^{\mathrm{f}}(M)\neq \emptyset $. The following are equivalent:

\begin{enumerate}
\item[(a)] $\mathrm{Spec}^{\mathrm{f}}(M)$ is irreducible;

\item[(b)] $\mathrm{Corad}^{\mathrm{f}}(M)$ is a first submodule of $M;$

\item[(c)] $0\neq \mathrm{Corad}^{\mathrm{f}}(M)$ is hollow \emph{(}uniserial%
\emph{)};

\item[(d)] $\mathrm{Spec}^{\mathrm{f}}(M)$ is a chain.
\end{enumerate}
\end{theorem}

\begin{notation}
Set%
\begin{equation}
\mathrm{Max}(\mathrm{Spec}^{\mathrm{f}}(M)):=\{K\in \mathrm{Spec}^{\mathrm{f}%
}(M)\mid \text{ }K\text{ is a maximal first submodule of }M\}.
\end{equation}
\end{notation}

\begin{prop}
\label{max-irr}Let $_{R}M$ be a consistent top$^{\mathrm{f}}$-module.

\begin{enumerate}
\item[(a)] We have a bijection%
\begin{equation}
\mathrm{Spec}^{\mathrm{f}}(M)\overset{V(-)}{\longleftrightarrow }\{{\mathcal{%
A}}\mid {\mathcal{A}}\subseteq \mathrm{Spec}^{\mathrm{f}}(M)\text{ is an
irreducible closed subset}\}.  \label{s-irr-closed}
\end{equation}

\item[(b)] The bijection \emph{(\ref{s-irr-closed})} restricts to a
bijection
\begin{equation*}
\mathrm{Max}(\mathrm{Spec}^{\mathrm{f}}(M))\overset{V(-)}{%
\longleftrightarrow }\{{\mathcal{A}}\mid {\mathcal{A}}\subseteq \mathrm{Spec}%
^{\mathrm{f}}(M)\text{ is an irreducible component}\}.
\end{equation*}
\end{enumerate}
\end{prop}

\begin{proof}
(a) Let $K\in \mathrm{Spec}^{\mathrm{f}}(M).$ Notice that $K=I(V(K))$ and so
the closed set $V(K)=\{K\}$ is irreducible (see Proposition \ref{duo-irr}).
On the other hand, let ${\mathcal{A}}\subseteq \mathrm{Spec}^{\mathrm{f}}(M)$
be a closed irreducible subset. Notice that $I({\mathcal{A}})$ is first in $%
M $ by Proposition \ref{duo-irr} and that ${\mathcal{A}}=\overline{{\mathcal{%
A}}}=V(I({\mathcal{A}})).$ Clearly, the maps $V$ and $I$ are bijective and
the result follows.

(b) This follows from (a), the definitions and the fact that $V$ is order
preserving.
\end{proof}

\begin{cor}
If $_{R}M$ is a consistent top$^{\mathrm{f}}$-module, then $\mathrm{Spec}^{%
\mathrm{f}}(M)$ is a sober space.
\end{cor}

\begin{proof}
Let ${\mathcal{A}}\subseteq \mathrm{Spec}^{\mathrm{f}}(M)$ be an irreducible
closed subset. By Proposition \ref{max-irr} (1), ${\mathcal{A}}=V(K)$ for
some $K\in \mathrm{Spec}^{\mathrm{f}}(M)$. It follows that
\begin{equation*}
{\mathcal{A}}=\overline{{\mathcal{A}}}=V(I({\mathcal{A}}))=V(K)=\overline{%
\{K\}},
\end{equation*}%
\emph{i.e.} $K$ is a generic point for ${\mathcal{A}}$. If $H$ is a generic
point of ${\mathcal{A}}$, then $V(K)=V(H)$ whence $K=H.$
\end{proof}

\begin{theorem}
\label{compact}Let $_{R}M$ be a top$^{\mathrm{f}}$-module with essential
socle.

\begin{enumerate}
\item[(a)] If $\mathcal{S}(M)$ is finite, then $\mathrm{Spec}^{\mathrm{f}%
}(M) $ is compact.

\item[(b)] If $\mathcal{S}(M)$ is countable, then $\mathrm{Spec}^{\mathrm{f}%
}(M)$ is countably compact.
\end{enumerate}
\end{theorem}

\begin{proof}
We prove only (a); the proof of (b)\ is similar. Assume that $\mathcal{S}%
(M)=\{N_{1},\cdots ,N_{k}\}.$ Let $\{V(H_{\alpha })\}_{\alpha \in I}$ be an
arbitrary collections of closed subsets of $\mathrm{Spec}^{\mathrm{f}}(M)$
with $\dbigcap\limits_{\alpha \in I}V(H_{\alpha })=\emptyset $. Since $%
\mathcal{S}(M)\subseteq \mathrm{Spec}^{\mathrm{f}}(M),$ we can pick for each
$i=1,\cdots ,k$ some $\alpha _{i}\in I$ such that $N_{i}\nsubseteqq
H_{\alpha _{i}}.$ If $\widetilde{H}:=\dbigcap\limits_{i=1}^{k}H_{\alpha
_{i}}\neq 0,$ then there exists a simple $R$-submodule $0\neq N\subseteq
\widetilde{H}$ (since $\mathrm{Soc}(\widetilde{H})=\widetilde{H}\cap \mathrm{%
Soc}(M)\neq 0$), a contradiction since $N=N_{i}\nsubseteqq H_{\alpha _{i}}$
for some $i=1,\cdots ,n$. It follows that $\widetilde{H}=0,$ whence $%
\dbigcap\limits_{i=1}^{k}V(H_{\alpha
_{i}})=V(\dbigcap\limits_{i=1}^{k}H_{\alpha _{i}})=V(0)=\emptyset .$
\end{proof}

\subsection*{Connectedness Properties}

Recall (e.g. \cite{Bou1966}, \cite{Bou1998}) that a non-empty topological
space $\mathbf{X}$ is said to be

\emph{ultraconnected}, iff the intersection of any two non-empty closed
subsets is non-empty;

\emph{irreducible} (or \emph{hyperconnected}), iff $\mathbf{X}$ is not the
union of two proper \textit{closed} subsets, or equivalently iff the
intersection of any two non-empty open subsets is non-empty;

\emph{connected}, iff $\mathbf{X}$ is not the \textit{disjoint} union of two
proper \textit{closed} subsets; equivalently, iff the only subsets of $%
\mathbf{X}$ that are \textit{clopen} (\emph{i.e.} closed and open) are $%
\emptyset $ and $\mathbf{X}$.

\begin{prop}
\label{it-irr}Let $_{R}M$ be a top$^{\mathrm{f}}$-module and assume that
every first submodule of $M$ is simple.

\begin{enumerate}
\item[(a)] $\mathrm{Spec}^{\mathrm{f}}(M)$ is discrete.

\item[(b)] $M$ has a unique simple $R$-submodule if and only if $\mathrm{Spec%
}^{\mathrm{f}}(M)$ is connected.

\item[(c)] $_{R}M$ is colocal if and only if $\mathrm{Spec}^{\mathrm{f}}(M)$
is connected and $\mathrm{Soc}(M)\leq _{R}M$ is essential.
\end{enumerate}
\end{prop}

\begin{proof}
(a) Notice that $_{R}M$ has the min-property by Corollary \ref{simple-0} and
Remark \ref{min}. It follows that for every $K\in \mathrm{Spec}^{\mathrm{f}%
}(M)={\mathcal{S}}(M)$ we have $\{K\}={\mathcal{X}}(\{K\}_{e})$ an open set.

(b) ($\Rightarrow $) clear.

($\Leftarrow $) By (a), $\mathrm{Spec}^{\mathrm{f}}(M)$ is discrete and so ${%
\mathcal{S}}(M)=\mathrm{Spec}^{\mathrm{f}}(M)$ has only one point since a
discrete connected space cannot contain more than one-point.

(c) follows directly from the definitions and (b).
\end{proof}

\begin{rem}
\label{s-t1}Let $_{R}M$ be a top$^{\mathrm{f}}$-module with essential socle.
Recall that $\mathcal{S}(M)\subseteq \mathrm{Spec}^{\mathrm{f}}(M)$ without
any conditions on $_{R}M.$ If $\{H\}$ is closed in $\mathrm{Spec}^{\mathrm{f}%
}(M)$ for some $H\leq _{R}M,$ then $\{H\}=V(K)$ for some $0\neq K\leq _{R}M$
and we conclude that $_{R}H$ is simple: if not, then there exists some
simple $R$-submodule $\widetilde{H}\lvertneqq _{R}H$ and we would have $\{H,%
\widetilde{H}\}\subseteq V(K)=\{H\}$, a contradiction. So, $H\leq _{R}M$ is
simple if and only if $H$ is a first submodule of $M$ and $V(H)=\{H\}$ if
and only if $\{H\}$ is closed in $\mathrm{Spec}^{\mathrm{f}}(M)$. Assume that
\end{rem}

\qquad Combining Proposition \ref{it-irr} and Remark \ref{s-t1} we obtain

\begin{theorem}
\label{T2}For a top$^{\mathrm{f}}$-module $M$ with essential socle, the
following are equivalent:

\begin{enumerate}
\item $\mathrm{Spec}^{\mathrm{f}}(M)=\mathcal{S}(M);$

\item $\mathrm{Spec}^{\mathrm{f}}(M)$ is discrete;

\item $\mathrm{Spec}^{\mathrm{f}}(M)$ is $T_{2}$ \emph{(}Hausdorff space%
\emph{)};

\item $\mathrm{Spec}^{\mathrm{f}}(M)$ is $T_{1}$ \emph{(}Fr\'{e}cht space%
\emph{)}.
\end{enumerate}
\end{theorem}

\begin{prop}
\label{com-prop}Let $_{R}M$ be comultiplication.

\begin{enumerate}
\item[(a)] $_{R}M$ is a strongly top$^{\mathrm{f}}$-module; in particular, $%
_{R}M$ is a top$^{\mathrm{f}}$-module.

\item[(b)] $\mathcal{S}(M)=\mathrm{Spec}^{\mathrm{f}}(M),$ \emph{i.e. }every
first submodule of $M$ is simple.

\item[(c)] $\mathrm{Spec}^{\mathrm{f}}(M)$ is discrete.
\end{enumerate}
\end{prop}

\begin{proof}
Let $_{R}M$ be comultiplication.

(a) This follows directly from the fact that $\mathrm{Sub}_{c}(M)=\mathrm{Sub%
}(M),$ Equation (\ref{L1+L2}) (see Remark \ref{strong->top}).

(b) This follows from Lemma \ref{c-ann} (a) and the fact that all submodules
of a comultiplication module are also comultiplication.

(c) This follows from Proposition \ref{it-irr} (a).
\end{proof}

\begin{ex}
\label{com-s}If $R$ is a \emph{left dual ring} \cite{NY2003}, then $_{R}R$
is a strongly top$^{\mathrm{f}}$-module and $\mathrm{Spec}^{\mathrm{f}%
}(_{R}R)=\mathrm{Min}(_{R}R)$ the set of minimal left ideals of $R.$
\end{ex}

\begin{ex}
\label{Prufer}$\mathbb{Z}_{p^{\infty }}$ is a comultiplication $\mathbb{Z}$%
-module, whence a strongly top$^{\mathrm{f}}$-module. Any $\mathbb{Z}$%
-submodule of $\mathbb{Z}_{p^{\infty }}$ is of the form $\mathbb{Z}(\frac{1}{%
p^{n}}+\mathbb{Z})$ for some $n\in \mathbb{N}$ and so $\mathbb{Z}_{p^{\infty
}}\notin \mathrm{Spec}^{\mathrm{f}}(\mathbb{Z}_{p^{\infty }})$ since $%
\mathrm{ann}_{\mathbb{Z}}(\mathbb{Z}_{p^{\infty }})=0\neq \mathrm{ann}_{%
\mathbb{Z}}(\mathbb{Z}(\frac{1}{p^{n}}+\mathbb{Z}))$ for every $n\in \mathbb{%
N}.$ Moreover, it is evident that $\mathrm{ann}_{\mathbb{Z}}(\mathbb{Z}(%
\frac{1}{p^{n_{1}}}+\mathbb{Z}))\varsupsetneqq \mathrm{ann}_{\mathbb{Z}}(%
\mathbb{Z}(\frac{1}{p^{n_{2}}}+\mathbb{Z})),$ whence $\mathbb{Z}(\frac{1}{%
p^{n_{2}}}+\mathbb{Z})\notin \mathrm{Spec}^{\mathrm{f}}(\mathbb{Z}%
_{p^{\infty }})$ if $n_{1}\lvertneqq n_{2}.$ Consequently, $\mathrm{Spec}^{%
\mathrm{f}}(\mathbb{Z}_{p^{\infty }})=\{\mathbb{Z}(\frac{1}{p}+\mathbb{Z})\}=%
\mathcal{S}(\mathbb{Z}_{p^{\infty }}).$ Clearly, $\tau ^{\mathrm{f}}(\mathbb{%
Z}_{p^{\infty }})=\{\emptyset ,\{\mathbb{Z}(\frac{1}{p}+\mathbb{Z})\}\}$ is
the trivial topology and is connected.
\end{ex}

\begin{prop}
\label{uniform}A top$^{\mathrm{f}}$-module $M$ with essential socle is
uniform if and only if $\mathrm{Spec}^{\mathrm{f}}(M)$ is ultraconnected.
\end{prop}

\begin{proof}
$(\Rightarrow )$ Let $_{R}M$ be uniform. For any non-empty closed subsets $%
V(K_{1}),V(K_{2})\subseteq \mathrm{Spec}^{\mathrm{fc}}(M),$ we have indeed $%
H_{1}\neq 0\neq H_{2}$ whence $V(H_{1})\cap V(H_{2})=V(H_{1}\cap H_{2})\neq
\emptyset ,$ since $H_{1}\cap H_{2}\neq 0$ by uniformity of $_{R}M$ and so
it contains by assumption some simple $R$-submodule which is indeed first in
$M.$

$(\Leftarrow )$ Assume that the $\mathrm{Spec}^{\mathrm{f}}(M)$ is
ultraconnected. Let $H_{1}$ and $H_{2}$ be non-zero $R$-submodules of $%
_{R}M. $ It follows that $V(H_{1})\neq \emptyset \neq V(H_{2}).$ By
assumption, $V(H_{1}\cap H_{2})=V(H_{1})\cap V(H_{2})\neq \emptyset ,$ hence
$H_{1}\cap H_{2}\neq 0.$
\end{proof}

\end{document}